\documentclass[12pt,thmsa, leqno]{amsart}

\usepackage{mathrsfs}
\usepackage{amsfonts}
\usepackage{amssymb}
\usepackage{amsmath}

\usepackage[dvips]{graphics}
\input{amssym.def}
\input{amssym.tex}

\def\th{\theta}

\def\Cal{\mathcal}

\def\P{{\Cal P}}

\def\I{{\Cal I}}

\def\Stab{{\hbox{\rm Stab}}}

\def\f0{f_0}
\def\Fc0{\varphi_0}

\def\I_k {I_{-}^{k/2}}
\def\I+k {I_{+}^{k/2}}

\def\bbr{{\Bbb R}}

\def\bbn{{\Bbb N}}
\def\bbh{{\Bbb H}}

\def\bbe{{\Bbb E}}
\def \hn{{\bbh^n}}

\def\diag{{\hbox{\rm diag}}}

\def\cosh{{\hbox{\rm cosh}}}
\def\sinh{{\hbox{\rm sinh}}}

\def\ch{{\hbox{\rm cosh}}}
\def\sh{{\hbox{\rm sinh}}}

\def\part{\partial}
\def\intl{\int\limits}
\def\b{\beta}

\def\Gam{\Gamma}
\def\Om{\Omega}
\def\a{\alpha}

\def\Del{\Delta}

\def\vp{\varphi}

\def\gam{\gamma}

\def\sig{\sigma}

\def\z{\zeta}
\def\e{\varepsilon}

\font\frak=eufm10

\def\fr#1{\hbox{\frak #1}}

\def\frH{\fr{H}}

\def\part{\partial}
\def\intl{\int\limits}
\def\b{\beta}

\def\Gam{\Gamma}
\def\Om{\Omega}
\def\a{\alpha}

\def\hn{\bbh^n}

\newtheorem{theorem}{Theorem}[section]
\newtheorem{lemma}[theorem]{Lemma}
\newtheorem{definition}[theorem]{Definition}

\newtheorem{proposition}[theorem]{Proposition}

\theoremstyle{remark}
\newtheorem{remark}[theorem]{Remark}

\newtheorem{example}[theorem]{Example}

\numberwithin{equation}{section}

\newcommand{\be}{\begin{equation}}
\newcommand{\ee}{\end{equation}}

\newcommand{\bea}{\begin{eqnarray}}
\newcommand{\eea}{\end{eqnarray}}
\newcommand{\Bea}{\begin{eqnarray*}}
\newcommand{\Eea}{\end{eqnarray*}}

\def\sideremark#1{\ifvmode\leavevmode\fi\vadjust{\vbox to0pt{\vss
 \hbox to 0pt{\hskip\hsize\hskip1em
\vbox{\hsize2cm\tiny\raggedright\pretolerance10000
 \noindent #1\hfill}\hss}\vbox to8pt{\vfil}\vss}}}%

\begin{document}

\title [Radon Transforms]{Radon Transforms over Lower-Dimensional\\  Horospheres in Real  Hyperbolic Space
}


\author{ W.O. Bray and B. Rubin }
\address{Department of Mathematics,
Missouri State University, USA}
\email{WBray@MissouriState.edu}

\address{Department of Mathematics, Louisiana State University, Baton Rouge,
Louisiana 70803, USA}
\email{borisr@math.lsu.edu}





\begin{abstract}

We study horospherical Radon transforms that integrate  functions on the $n$-dimensional real hyperbolic space over horospheres of arbitrary fixed dimension $1\le d\le n-1$. Exact existence conditions and new explicit  inversion formulas  are obtained for these transforms acting on  smooth functions and functions belonging to $L^p$. 
The case $d=n-1$ agrees with the well-known Gelfand-Graev transform.

\end{abstract}




\maketitle

\section{Introduction}
\label{intro}
\setcounter{equation}{0}

Let $\bbh^n$ be  the $n$-dimensional real hyperbolic space.
 We will be dealing
 with the hyperboloid model of this space,  when  $\hn$ is identified with the upper sheet of the two-sheeted hyperboloid in the pseudo-Euclidean space $E^{n, 1} \sim \bbr^{n+1}$.
 The term {\it horosphere}   (or {\it orisphere}), which means  a sphere of infinite radius, was introduced by  Lobachevsky.
   In the hyperboloid model, the $(n-1)$-dimensional horosphere 
 is a cross-section  of the hyperboloid $\hn$ by the hyperplane whose normal lies in the  asymptotic
cone.

In the present article, we study  horospherical  Radon-like transforms $f \to \hat f$ that integrate functions on $\bbh^n$  over $d$-dimensional horospheres for  arbitrary $1\le d\le n-1$. Our main objective is explicit definition of these transforms, their properties, and  inversion formulas  on $L^p$ and smooth functions.

  In the case $d=n-1$, the corresponding  horospherical transforms are also known as the  {\it Gelfand-Graev transforms}; see
\cite[p. 290]{GGV},  \cite[p. 532]{V}, \cite[p. 162]{VK}. In these publications, a compactly supported smooth function $f$ was reconstructed from $\hat f$   in terms of certain  integrals that should be understood in the sense of distributions.

Our approach essentially differs from that in \cite{GGV, V, VK}   and consists of two parts. The first part deals with arbitrary continuous and $L^p$ functions and relies on the properties of some mean value operators. This idea  dates back to the classical works by Funk,  Radon, and Helgason; see historical notes in \cite{H11,  Ru15}. We show that it is applicable to horospherical transforms over horospheres of arbitrary dimension.

The second part deals with
 compactly supported smooth functions. Here the reconstruction of $f$ reduces to inversion of certain operators of the potential type by means of polynomials of the Beltrami-Laplace operator. Operators of this kind are hyperbolic counterparts of the classical Riesz potentials (see, e.g., Stein \cite{St1}) and might be of independent interest.  This approach was applied by Helgason  to totally geodesic Radon transforms of smooth functions on constant curvature spaces and extended by Rubin \cite{Ru15} to horospherical  transforms over $(n-1)$-dimensional horospheres.

It was surprising that Radon transforms over lower-dimensional horospheres in the hyperbolic space were not considered in the literature (to the best of our knowledge). Our aim is to complete this gap.

It is worth noting that horospherical  transforms play an important role in the representation theory and  appear in the general context of symmetric spaces under the name  {\it  ``horocycle transform''}. More information on this subject can be found in the works by Gelfand and Graev  \cite{GG59},
Helgason \cite{H73, H08, H11, H12}, Gindikin \cite{Gi01, Gi08, Gi13},  Gonzalez \cite{Go10a},  Gonzalez and Quinto \cite{GQ94}, Hilgert,  Pasquale, and  Vinberg  \cite{HPVa, HPV}; see also   Berenstein and Casadio Tarabusi \cite{BC94},   Bray and Rubin \cite{Bru} for the case $d=n-1$.  The methods and results of these publications essentially differ from those in the  present article.

{\bf Plan of the paper.} Section 2 contains auxiliary facts related to analysis on $\hn$. In Section 3 we define the horospherical transform $ \hat f (\xi)$ for $d$-dimensional horospheres $\xi$. In particular, we show that if $f\in L^p (\hn)$, then $\hat f (\xi)$ is finite for almost all $\xi$ provided $1\le p< 2(n-1)/d$ and this bound is sharp.
Section 4 is devoted to inversion formulas for $\hat f$ on  functions $f\in C_c^\infty (\hn)$ and $f\in L^p (\hn)$. 
The main  results are stated in Theorems  \ref{jOOOthERC2}, \ref{ThHORYP},  and \ref{ThHORYP1}.  We conclude the paper by Section 5, in which some open problems are formulated.

{\bf Acknowledgements.} After the manuscript had been  written, Professor Helgason kindly informed us that close problems, related to generalizations of the horospherical (or horocycle) transforms, and, in particular, to their lower-dimensional modifications in the general context of symmetric spaces were independently studied by M. Morimoto \cite{Mor1, Mor2} and E. Kelly \cite{Kel}. The authors are deeply grateful to
Sigurdur Helgason and Simon Gindikin for  correspondence. Special thanks go  to  Edmund Kelly who kindly sent us his 1974 preprint \cite{Kel}. The methods of our paper and the results are essentially different from those in \cite{Kel, Mor1, Mor2}.

\section{Preliminaries} \label{sec:1}

\subsection{Basic Definitions}\label {PPMMBRE}

The  pseudo-Euclidean space $\bbe^{n, 1}$, $n\ge 2$, is  the  $(n+1)$-dimensional real vector space of points
in $\bbr^{n +1}$ with the inner product
\be\label {tag 2.1-HYP}[{\bf x}, {\bf y}] = - x_1 y_1 - \ldots -x_n y_n + x_{n +1} y_{n +1}. \ee

We denote by $ \ e_1, \ldots, e_{n +1}$  the coordinate unit vectors  in $\bbe^{n,1}$;
$S^{n -1}$ is  the unit sphere  in the coordinate plane $\bbr^n=\{x\in \bbe^{n,1} : x_{n +1}=0\} $; $\sigma_{n-1} =  2\pi^{n/2} \big/ \Gamma (n/2)$ is the
surface area of $S^{n-1}$. For    $\theta \in S^{n-1}$, $d\theta$ denotes the surface element on $S^{n-1}$;
$d_*\theta= d\theta/\sigma_{n-1}$ is the normalized surface element on $S^{n-1}$ (a similar notation will be used for normalized surface elements of lower-dimensional spheres).

 Given a set $X$ and a group $G$,  the group action of $G$ on $X$ is a function
$G\times X \to X$ with $(g,x) \to gx$ where $g\in G$ and $x\in X$. We also write $Gx$ for the set $\{y\in X: y=gx \; \text{\rm for some $g\in G$}\}$; cf. Knapp \cite [p. 159]{Kn}.

 The $n$-dimensional real hyperbolic space $\hn$ is realized as the upper sheet of the two-sheeted hyperboloid  in $\bbe^{n, 1}$, that is,
\[\hn = \{{\bf x}\in \bbe^{n,1} :
[{\bf x}, {\bf x}] = 1, \ x_{n +1} > 0 \}.\]
In the following, the points of $\hn$  will be denoted by the non-boldfaced letters, unlike the generic points in $\bbe^{n,1}$.
The point $x_0=(0, \ldots, 0,1)\sim e_{n +1}$ serves as the origin of $\hn$;
\[ \Gam =\{{\bf x} \in \bbe^{n,1}: \, [{\bf x}, {\bf x}]=0, \, x_{n+1} >0\}.\]
is the asymptotic cone for $\hn$. The notation
\[G=SO_0(n,1)\]
is used for the identity component of the special pseudo-orthogonal group $SO(n,1)$  preserving the bilinear form $[{\bf x},  {\bf y}]$.

The geodesic distance between the points $x$ and $y$ in $\hn$ is defined by $d(x,y) = \cosh^{-1}[x,y]$, so that
\[ [x,a]=\ch r\]
is the equation of the $(n-1)$-dimensional geodesic sphere in $\hn$ of radius $r$ with center at $a\in\hn$.

We will be using different coordinate systems on $\hn$. Every  point
$x  \in \hn$ is represented in the {\it hyperbolic coordinates} $(\th, r)\in S^{n -1} \times [0,\infty)$ as
\be\label {taddd-HYP}  x = \theta\, \sh r  + e_{n+1} \, \ch r.\ee
In the {\it horospherical coordinates} $(v,t)\in  \bbr^{n-1} \times \bbr$, we have
\bea
\label{lol} x&=&n_{v}a_{t}x_0=a_{t}n_{e^{-t}v}x_0\\
\label{horo coord} &=&(e^{-t}v,\,\sinh t+\frac{|v|^{2}}{2}e^{-t},\,\cosh t+\frac{|v|^{2}}{2}e^{-t}).
\eea
Here  $v\in\bbr^{n-1}$ (a column vector),
\be\label{09m} n_{v}\!=\!\left[
\begin{array}
[c]{ccc}%
I_{n-1} & -v & v\\
v^{T} & 1\!-\!|v|^{2}/2 & |v|^{2}/2\\
v^{T} & -|v|^{2}/2 & 1\!+\!|v|^{2}/2
\end{array}
\right]; \ee
\be\label{mo3a1}
a_{t}=\left[
\begin{array}
[c]{ccc}%
I_{n-1} & 0 & 0\\
0 & \cosh t & \sinh t\\
0 & \sinh t & \cosh t
\end{array}
\right], \qquad t\in \bbr,\ee
is the hyperbolic rotation in the coordinate plane $(x_n, x_{n+1})$;  cf. \cite [p. 13] {VK}.
Changing variable $e^{-t}v \to v$, we also have
\be\label{horo coord1}
x=a_{t}n_{v}x_0 =n_{e^{-t}v}a_{t}x_0=(v,\,\sinh t+\frac{|v|^{2}}{2}e^{t},\,\cosh t+\frac{|v|^{2}}{2}e^{t}).
\ee
 A straightforward matrix multiplication yields
\be\label{ikm}  n_{v_1} n_{v_2}=  n_{v_1  +v_2}\quad \text {\rm for all $\; v_1, v_2 \in \bbr^{n-1}$}.\ee

Let
\be\label {o8a3}
 K\!=\!\left \{\left[
\begin{array}
[c]{cc}%
k & 0\\
0 & 1
\end{array}
\right] \, : k\in SO(n)\right\}.\ee
Abusing notation, we identify $ k\in SO(n)$ with the corresponding matrix $\left[
\begin{array}
[c]{cc}%
k & 0\\
0 & 1
\end{array}
\right]\in G$.
The subgroups of matrices $a_t$ and $n_v$ will be denoted by $A$ and $N$, respectively.
Since the stabilizer of $x_0$ in $G$ coincides with $K$, (\ref{lol}) implies that every $g\in G$ is representable in the form
\be\label{yuyy} g= n_v a_t k. \ee
This representation is unique and  agrees with the Iwasawa decomposition $G=NAK$.

We fix a $G$-invariant measure $dx$ on $\hn$, which has the following form in the coordinates (\ref{taddd-HYP}):
\be\label {kUUUPqs}  d x = \sh^{n -1} r \, d r d \theta.\ee
If $f$ is $K$-invariant, that is,    $f(x)\equiv f_0 (x_{n+1})$, then
\be\label{ppooii}
\intl_{\hn} f(x)\, dx=\sig_{n-1}\intl_1^\infty f_0(s) (s^2 -1)^{n/2 -1}\, ds.\ee

 The Haar measure $dg$ on $G$ will be normalized in a consistent way by the formula
 \be\label {tag 2.3-AIM}
\intl_G f(gx_0)\,dg=\intl_{\hn} f(x)\,dx.\ee
Using (\ref{yuyy}), we also have
\be\label{frg} dg= e^{(1-n)t}  dt dv dk,\ee
where $dk$ is the normalized Haar measure on $K$, $dv$ and $dt$ are the standard Euclidean measures on $\bbr^{n-1}$ and $\bbr$, respectively; cf. \cite[p. 23]{VK}. Thus,
\be\label{frg1}
\intl_G f(g x_0)\,dg=\intl_\bbr e^{(1-n)t}\, dt \intl_{\bbr^{n-1}} f(n_v a_t x_0)\,dv=\intl_{\hn} f(x)\,dx \ee
or, by (\ref{lol}),
\be\label{frg12}
\intl_G f(g x_0)\,dg=\intl_\bbr dt \intl_{\bbr^{n-1}} f(a_t n_v  x_0)\,dv=\intl_{\hn} f(x)\,dx; \ee
cf. \cite[Lemma 3.1]{Ru17}.
The equality (\ref{frg12}) agrees with the representation
\be\label{yu} g= a_t n_v  k, \ee
in terms of which $dg=dt dv dk$ and $G=ANK$.

Replacing $g$ by $kg$, $k\in K$, in (\ref{frg12}), we obtain
\be\label{frg1a2s}
\intl_G f(g x_0)\,dg=\intl_K dk\intl_\bbr dt \intl_{\bbr^{n-1}} f(ka_t n_v  x_0)\,dv=\intl_{\hn} f(x)\,dx. \ee
This equality (\ref{frg1a2s}) agrees with the representation
\be\label{yu} g= ka_t n_v, \ee
in terms of which $dg=dkdtdv $ and $G=KAN$.

{\bf More notation.} In the following,
  $ u\cdot v  =u_1 v_1 + \ldots + u_n v_n$  denotes the
usual inner product of the vectors $u, v \in \bbr^n$; $I_m$ is the identity $m\times m$ matrix;
$C(\hn)$ is the
space of continuous functions on $\hn$; $C_0 (\hn)$ denotes the space of continuous functions
on $\hn$ vanishing at infinity.
We also set
\be\label{benatur} C_\mu (\hn)=\{f\in C(\hn): f(x)= O(x_{n+1}^{-\mu})\}.\ee

Let $\Om=\{{\bf x}\in E^{n,1}: [{\bf x}, {\bf x}]>0, \,x_{n+1} >0\}$ be
 the interior of  the cone $\Gam$.
 We denote by
  $ C^\infty_c(\hn)$ the space of  infinitely differentiable compactly supported  functions on $\hn$. This space is formed by the restrictions onto $\hn$ of functions belonging to $C^\infty_c (\Om)$.

We say that  an integral under consideration
 exists in the Lebesgue sense if it is finite when the corresponding integrand is replaced by its absolute value.  The letter $c$ (sometimes with subscripts) denotes  a constant that may vary at each  occurrence.

\subsection{Horospheres}

\subsubsection{  The case $d=n-1$}

An $(n-1)$-dimensional
 horosphere  in $\hn$ is defined as the cross-section of the hyperboloid  $\hn$ by the hyperplane $[{\bf x}, {\bf b}]=1$, where
$ {\bf b}$  is a point of the cone $\Gam$.  The  correspondence between the set $\Xi_{n-1}$ of all $(n-1)$-horospheres  and the set of all points in $\Gam$  is  one-to-one. One can equivalently define  $\Xi_{n-1}$ as the set of all $G$-orbits
\be\label {o8} \Xi_{n-1}=\{g\xi^0_{n-1}: \; g\in G\}\ee
of the ``basic'' horosphere $\xi^0_{n-1}$ corresponding to the point
\[  {\bf b}_{0}=(0,\ldots,0,1,1)\in \Gam.\]
The stabilizer $G^0_{n-1}$ of ${\bf b}_{0}$ (and therefore, of $\xi^0_{n-1}$) in $G$ is the semidirect product
\be\label{mo1pp}
G^0_{n-1}=N \rtimes  M\ee
where  $N$ is the group of transformations (\ref{09m}) and
\[ M\!=\!\left \{\left[
\begin{array}
[c]{cc}%
m & 0\\
0 & I_2
\end{array}
\right] \, : m\in SO(n-1)\right\}.\]
We observe that $N$ is a normal subgroup of $G^0_{n-1}$.
For the sake of simplicity, we write (\ref{mo1pp}) as
\be\label{mo1pp1}
G^0_{n-1}=\Stab_G (\xi^0_{n-1})=\Stab_G ({\bf b}_{0})= MN\ee
(cf. \cite[p. 60]{H08}). Note also that $MN=NM$.

\subsubsection{  The case $d<n-1$}

We set
 \be\label{mo1} \bbr^{n-d-1}\!=\!\bbr e_1 \oplus \cdots \oplus\bbr e_{n-d-1} \quad \text{\rm (if $d \!= \!n\!-\!1$ this set is empty)},\ee
\be\label{mo2} \bbr^{d}=\bbr e_{n-d} \oplus \cdots \oplus\bbr e_{n-1}, \quad  \bbr^{d+1}=\bbr^{d} \oplus\bbr e_{n},\ee
\be\label{mo3}  \bbe^{d+1,1}\simeq  \bbr^{d+2}=\bbr^{d+1} \oplus\bbr e_{n+1}, \qquad  \bbh^{d+1}=\hn \cap \bbr^{d+2};\ee
\be\label{bo3}
 \xi_0= \bbh^{d+1} \cap \{x \in \hn: [x,{\bf b}_0]=1\}, \qquad {\bf b}_{0}=(0,\ldots,0,1,1).
 \ee
The last formula defines a $d$-dimensional  horosphere in $\bbh^{d+1}$. We call it the basic one.
The set $\Xi_{d}$ of $d$-dimensional horospheres in $\hn$ ($d$-horospheres, for short) is defined as the collection of all $G$-orbits
\be\label {o8a} \Xi_{d}=\{g\xi_0: \; g\in G\}.\ee

Let $M_d\subset G$ be the subgroup of matrices of the form
\[
 m_{\a,\b}=\left[
\begin{array}
[c]{cc}%
\tilde m_{\a,\b} & 0\\
0 & I_2
\end{array}
\right]\]
where
\[
\tilde m_{\a,\b} =\left[
\begin{array}
[c]{cc}%
\a & 0\\
0 & \b
\end{array}
\right]\in S(O(n-d-1)\times O(d)).\]
Let also
 \be\label {o8a1}
 N_d=\{n_v\in N:  v_1= \ldots =v_{n-1-d}=0\},
 \ee
 \be\label {o8a1a}
 N_{n-1-d}=\{n_v\in N:  v_d= \ldots =v_{n-1}=0\},
 \ee
  be the subgroups of $N$ (cf. (\ref{09m})) generated by vectors $v\in \bbr^{n-1}$ with the corresponding  zero coordinates.
  A straightforward matrix multiplication yields
\be\label {o8s5}
 m_{\a,\b} n_v=n_{\b v} m_{\a,\b}\ee
 for all $m_{\a,\b} \in M_d$ and $n_v \in N_d$,
 so that we can write $M_d N_d=N_d M_d$.  We denote
 \be\label {o8135}
G_d^0= M_d N_d=N_d M_d.\ee

\begin{proposition}\label{plu} ${}$\hfill

\noindent {\rm (i)} The basic $d$-horosphere  $\xi_0$ is the $N_d$-orbit of $x_0$. Moreover,
 \be\label {o8135x}
 \xi_0=G_d^0 \,x_0.\ee

\noindent {\rm (ii)} The subgroup $G_d^0=M_d N_d$ is the stabilizer of $\xi_0$ in $G$, so that the
 set $\Xi_{d}$ of all $d$-horospheres in $\hn$ is isomorphic to the quotient space $G/G_d^0$.
\end{proposition}
\begin{proof} The first statement in (i) is a consequence of the similar fact for $d=n-1$. Then
 \[ G_d^0 \,x_0= N_d M_d \,x_0=N_d \,x_0=\xi_0.\]

To prove (ii), we observe that by (\ref{bo3}), $\xi_0$ can be identified with the pair $(\bbr^{n-1-d}, {\bf b}_{0})$. Thus it suffices to show that  $M_d N_d$ is the stabilizer of this pair, i.e.,
\be\label {o8a4y} \Stab_G  (\bbr^{n-1-d}, {\bf b}_{0})=M_d N_d.\ee
Because
\[ \Stab_G  (\bbr^{n-1-d}, {\bf b}_{0})=\Stab_G  (\bbr^{n-1-d}) \cap \Stab_G  ({\bf b}_{0})\]
and $\Stab_G  ({\bf b}_{0})=MN$ (see (\ref{mo1pp1})),
it remains to prove that
\be\label {o8a4y1} \Stab_{MN}  (\bbr^{n-1-d}) =M_d N_d.\ee
The embedding $M_d N_d \subset  \Stab_{MN}  (\bbr^{n-1-d})$ is obvious because  $m_{\a,\b}\in M_d$
 and $n_v  \in N_d$ preserve $\bbr^{n-1-d}$.
 Conversely, suppose that $g\in MN$ preserves $\bbr^{n-1-d}$ and let $g=\gam n_v$ with
 \[
 \gam=\left[
\begin{array}
[c]{cc}%
\tilde \gam & 0\\
0 & I_2
\end{array}
\right], \quad \tilde \gam =\left[
\begin{array}
[c]{cc}%
\a & \mu\\
\nu & \b
\end{array}
\right]\in SO(n-1),\]
and $v=(v_1,\ldots, v_{n-1})\in \bbr^{n-1}$.
If $e_j$ ($j=1, \ldots , n-1-d$) are coordinate unit vectors, then
\[
ge_j=\gam n_v e_j=\left[
\begin{array}
[c]{cccc}%
\a & \mu & 0 & 0\\
\nu & \b & 0 & 0\\
0 & 0 & 1 & 0\\
0 & 0 & 0 & 1
\end{array}
\right]\,
\left[
\begin{array}
[c]{c}%
e_j\\
0 \\
v_j\\
v_j
\end{array}
\right]=\left[
\begin{array}
[c]{c}%
\a e_j\\
\nu e_j \\
v_j\\
v_j
\end{array}
\right]. \]

The vector on the right-hand side belongs to $\bbr^{n-1-d}$ if and only if $\nu$ is a zero matrix  and $v_j=0$ for all $j=1, \ldots , n-1-d$. Thus, if $g\in MN$ preserves $\bbr^{n-1-d}$, then, necessarily,
$g=\left[
\begin{array}
[c]{cc}%
\tilde \gam & 0\\
0 & I_2
\end{array}
\right]\, n_v$ with $n_v\in N_d$ and $\tilde \gam=\left[
\begin{array}
[c]{cc}%
\a & \mu\\
0 & \b
\end{array}
\right]$.
Because $\tilde \gam \in SO(n-1)$, we have $\tilde \gam^T \tilde \gam=\tilde \gam \tilde \gam^T=I_{n-1}$. Multiplying matrices, we obtain
\[\a^T\a=I_{n-1-d}, \qquad \a^T \mu =0,\qquad \mu^T\mu +\b^T\b=I_d,\]
\[\a\a^T +\mu\mu^T=I_{n-1-d}, \qquad \mu \b^T=0,\qquad \b\b^T=I_d.\]
It follows that
\[\a^T\a=\a\a^T =I_{n-1-d}, \qquad b^T\b=\b\b^T=I_d,\]
and therefore, $\a\in O(n-1-d)$, $\b\in O(d)$. Since $\tilde \gam \in SO(n-1)$, we obtain $\tilde \gam=\diag (\a,\b)\in  S(O(n-d-1)\times O(d))$.
The latter means that $\Stab_{MN}  (\bbr^{n-1-d}) \subset M_d N_d$ which completes the proof.
\end{proof}

According to the Iwasawa decomposition $G=KAN$, every $g\in G$ is uniquely represented as $g=ka_tn_v$, where $k \in K$, $a_t \in A$, and $n_v \in N$.
We write $v \in \bbr^{n-1}$ as an orthogonal sum
\be\label {o8a4}  v=u+w, \qquad u \in \bbr^{n-1-d}, \quad w \in \bbr^{d}.\ee
Noting that $n_w \xi_0 =\xi_0$ for all $w \in \bbr^{d}$, we conclude that  every $d$-horosphere $\xi=g\xi_0$ can be represented as
\be\label {o8a5}
\xi=ka_tn_u \xi_0, \qquad k \in K, \quad a_t \in A, \quad n_u \in N_{n-1-d}.\ee
Following this equality, we equip $\Xi_d$ with the  measure $d\xi$ by setting
\be\label {o8a2q}
\intl_{\Xi_d} \vp (\xi)\, d\xi= \intl_K dk\intl_\bbr dt \intl_{\bbr^{n-1-d}} \vp(ka_t n_u  \xi_0)\,du.\ee

\begin{proposition}\label{pluc}  Let $\Xi_{d}$  be the set (\ref{o8a}) of $d$-horospheres in $\hn$ with the basic horosphere $\xi_0$ (the ``origin" of $\Xi_{d}$) and the stabilizer  $G_d^0=\Stab_G (\xi_0)$. The set $\,\Xi_{d}$ coincides with the set of orbits of conjugates of $G_d^0$. Specifically,
\be\label {o8va2q}
\Xi_{d}=\{gG_d^0 g^{-1} x: \; g\in G, \,x\in \hn\}.\ee
\end{proposition}
\begin{proof} Denote the right-hand side of (\ref{o8va2q}) by $\Xi'_{d}$.  Let $\xi\in \Xi_{d}$, that is, $\xi=g\xi_0$ for some $g\in G$. Setting $g=ka_tn_v=ka_tn_u n_w$ (cf. (\ref{o8a4})), we obtain
\bea
\xi&=&ka_tn_u n_w \xi_0=ka_tn_u  \xi_0= ka_tn_u  G_d^0 x_0\nonumber\\
&=&(ka_tn_u)\, G_d^0\, (ka_tn_u)^{-1}\, (ka_tn_u)\, x_0=\tilde G_d^0\, x,\nonumber\eea
where $\tilde G_d^0 =(ka_tn_u)\, G_d^0\, (ka_tn_u)^{-1}$ and $x=(ka_tn_u)\, x_0$. Hence $\xi\in \Xi'_{d}$.

Conversely, let $\xi\in  \Xi'_{d}$, that is, $\xi=g\,G_d^0 \,g^{-1} x$ for some $g\in G$ and  $x\in \hn$. We write $g^{-1} x$ in horospherical coordinates as $g^{-1} x=n_v a_t x_0$ (cf. (\ref{lol})). Then
 \bea
\xi&=&g\,G_d^0 \,n_v a_t x_0= g\,M_d N_d \,n_w n_u a_t x_0\nonumber\\
&=&g\,M_d N_d \, n_u a_t x_0=g\,M_d  n_u N_d \, a_t x_0. \nonumber\eea
By (\ref{lol}) and (\ref{horo coord1}),  $N_d \, a_t x_0= a_t N_d   x_0$. Hence
\[\xi=g\,M_d  n_u a_t N_d   x_0=g\,M_d  n_u a_t \xi_0 \subset G\xi_0 =\Xi_{d}.\]
\end{proof}

\begin{remark} {\rm
Proposition \ref{pluc} shows that our definition  of $d$-horospheres agrees with  Helgason's definition of horocycles in symmetric spaces; see \cite[p. 60]{H08}.}
\end{remark}

\section{ Definition and Basic Properties of the $d$-Horospherical Transform }

\begin{definition} \label {o8a5x} Let $1\le d\le n-1$. Given  $\xi=g\xi_0\in \Xi_d$, $g\in G$, the $d$-horospherical transform of a sufficiently good function $f$ on $\hn$ is defined by
\be\label{bo9}
\hat f (\xi) =\intl_{\bbr^d} f(gn_w x_0)\,dw.\ee
\end{definition}

\begin{remark} The definition can be put in group-theoretic terms as follows. Since $\mathbb{H}^n$ is identified as the homogeneous space $G/K$, a function $f$ on $\mathbb{H}^n$ becomes the right $K$-invariant function $f(gK)$ on $G$. Denoting group elements in $N_d$ by $n_d$ and the Haar measure on $N_d$ by $dn_d$, we can write the defining formula as
\be
\hat f (gM_d N_d)=\intl_{N_d} f(gn_d K)\,dn_d. \ee
\end{remark}

The case when $g$ is the identity map, corresponds to the integral of $f$ over $\xi_0$.  If $g=ka_tn_v=ka_tn_{u+w}$, then, by (\ref{ikm}),
\be\label{bo9n}
\hat f (\xi) \equiv \hat f (g\xi_0)=\intl_{\bbr^d} f(ka_tn_{u} n_{w}x_0)\,dw.\ee

By  (\ref{bo9}), the map $f \to \hat f$ is $G$-equivariant. Indeed, for all $\gam\in G$,
\bea \hat f (\gam \xi) &\equiv& \hat f (\gam g\xi_0)=\intl_{\bbr^d} f(\gam gn_w x_0)\,dw \nonumber\\
&=&\intl_{\bbr^d} (f\circ \gam)(gn_w x_0)\,dw =(f\circ \gam)^\wedge (g\xi_0)\equiv (f\circ \gam)^\wedge(\xi).\nonumber\eea

In particular, if $f$ is $K$-invariant  (or zonal), then so is $\hat f$. The  $d$-horospherical transform of a $K$-invariant function expresses through the Riemann-Liouville fractional integral
\be\label{rl-}
(I^\a_{-}\psi ) (r) = \frac{1}{\Gamma (\alpha)} \intl_r^{\infty}
\frac{\psi(s) \,ds} {(s-r)^{1- \alpha}}\, \qquad \a>0,\ee
as follows.

\begin{lemma} \label{zonal} Let $\vp (\xi)=\hat f (\xi)\equiv \hat f (ka_tn_{u}\xi_0)$,  $f(x)\equiv f_{0}(x_{n+1})$. Then
\be\label {o8a5br}
\vp (\xi)= c\, e^{-td/2}\,(I_-^{d/2} f_0)(\eta), \quad \eta= \cosh t + \frac{|u|^2}{2}\, e^t, \ee
where
\[ c=2^{d/2 -1}\sig_{d-1}\,\Gam(d/2).\]
It is assumed that the integrals in (\ref{o8a5br}) exist in the Lebesgue sense.
\end{lemma}

\begin{proof} Because $f$ is $K$-invariant, from (\ref{bo9}) and (\ref{horo coord1})  we have
\bea
\hat f(\xi)&=&\intl_{\bbr^d} f_0 \Big (\cosh t+
\frac{|u|^{2}+|w|^{2}}{2}\, e^{t} \Big )\, dw\nonumber\\
&=&\sig_{d-1}\intl_0^\infty f_0 \Big (\cosh t+
\frac{|u|^{2}+|r|^{2}}{2}\, e^{t} \Big )\, r^{d-1}\, dr.\nonumber\eea
This gives (\ref{o8a5br}).
\end{proof}

Our next goal is to establish conditions  under which  $\hat f$ exists as an absolutely convergent integral. We restrict our consideration to  two cases:  $f\in C_\mu (\hn)$ and $f\in L^p (\hn)$.

\begin{lemma} \label{zo} If $f\in L^1 (\hn)$, then  the integral (\ref{bo9n}) is finite for all $k\in K$, almost all $t\in \bbr$ and almost all $u\in \bbr^{n-1-d}$. Moreover, for all $k\in K$,
\be\label {o8a5btg} \intl_\bbr dt\intl_{\bbr^{n-1-d}} \hat f(ka_tn_u \xi_0)\, du=\intl_{\hn} f(x)\, dx  \quad \text{if $\;d<n-1$},\ee
\be\label {o8a5btg2} \intl_\bbr \hat f(ka_t \xi_0) \,dt=\intl_{\hn} f(x)\, dx \quad \text{if $\;d=n-1$}.\ee
\end{lemma}
\begin{proof} Let $d<n-1$. Then
\bea l.h.s.&=&\intl_\bbr dt\intl_{\bbr^{n-1-d}} du\intl_{\bbr^d} f(ka_tn_{u} n_{w}x_0)\,dw\qquad \text{\rm (set $u+w=v$)}\nonumber\\
&=&\intl_\bbr dt \intl_{\bbr^{n-1}} f(ka_tn_{v} x_0)\,dv\qquad \text{\rm (use (\ref{frg12}))}\nonumber\\
&=&\intl_{\hn} f(kx)\,dx=r.h.s. \nonumber\eea
If $d=n-1$, the proof is similar.
\end{proof}

\begin{proposition}  If $f\in C_\mu (\hn)$, $\mu>d/2$, then  $\hat f (\xi)$ is finite for every $\xi \in \Xi_d$.
\end{proposition}
\begin{proof}  By (\ref{bo9n}) and (\ref{horo coord1}),
\bea
|\hat f(\xi)|&\le& \intl_{\bbr^d} |f(ka_tn_{u} n_{w}x_0)|\, dv\le c\,  \intl_{\bbr^d} [a_tn_{u} n_{w}x_0, x_0]^{-\mu}\, dw\nonumber\\
&=&c\,\intl_{\bbr^d} \left[ \cosh t +\frac{|u|^2 + |w|^2}{2}\, e^t\right ]^{-\mu}\, dw.\nonumber\eea
The last integral is finite if $\mu>d/2$.
\end{proof}

\begin{remark} {\rm The condition $\mu>d/2$ is sharp. There is a function $f\in C_\mu (\hn)$, $\mu\le d/2$, for which $\hat f(\xi)\equiv \infty$. An example of such a function can be constructed by making use of  the Abel type representation (\ref{o8a5br}); see also Remark \ref{mjn}.}
\end{remark}

The question about the existence of $\hat f$ for $f\in L^p(\hn)$ requires some preparation.

\begin{lemma} \label{zo} If
\be\label{zop}
\intl_{y_{n+1}>1+\e} |f(y)|\, y_{n+1}^{d/2 -n+1}\, dy< \infty \quad \text {for some $\;\e>0$},\ee
then the integral $\hat f (\xi)\equiv \hat f(ka_tn_u \xi_0)$ is finite for almost all $k\in K$, almost all $|t|>\ch^{-1} (1+\e)$ and all $u\in \bbr^{n-1-d}$.
\end{lemma}
\begin{proof} We use the fact that the map $f \to \hat f$ is $K$-equivariant. Then
\[\intl_K \hat f(k\xi)\, dk= \hat f_z (\xi), \quad f_z(y)=\intl_K f(ky)\, dk \equiv f_0 (y_{n+1}).\]
 Hence,  Lemma \ref {zonal} yields
\[
\intl_K \hat f(k\xi)\, dk=c\, e^{-td/2}\,(I_-^{d/2} f_0)(\eta), \quad \eta= \cosh t + \frac{|u|^2}{2}\, e^t.\]
By Lemma  2.12 from \cite{Ru15},  the last integral is finite for almost all $\eta >1+\e$ provided
\[I_\e=\intl_{1+\e}^\infty s^{d/2 -1} |f_0(s)|\,ds <\infty.\]
However, by (\ref{ppooii}) and (\ref{zop}),
\bea
I_\e&=&\frac{1}{\sig_{n-1}}\intl_{y_{n+1}>1+\e}\frac{y_{n+1}^{d/2 -1}\, dy}{(y_{n+1}^2 -1)^{n/2 -1}} \intl_K f(ky)\, dk\nonumber\\
&\le& c_\e \intl_{y_{n+1}>1+\e} |f(y)|\, y_{n+1}^{d/2 -n+1}\, dy<\infty. \nonumber\eea
This completes the proof.
\end{proof}

Lemma \ref{zo} implies the following proposition that extends the existence result of Lemma \ref {zo} to $f\in L^p(\hn)$.

\begin{proposition} If $f\in L^p (\hn)$, $1\le p< 2(n-1)/d$, then  the integral  $\hat f (\xi)\equiv \hat f(ka_tn_u \xi_0)$ is finite for almost all $k\in K$, almost all $t\in \bbr$ and all $u\in \bbr^{n-1-d}$.
\end{proposition}
\begin{proof} By H\"older's inequality, the integral  (\ref{zop}) is dominated by $c_\e ||f||_p$ where
\[c_\e^{p'} =\intl_{y_{n+1}>1+\e}  y_{n+1}^{(d/2 -n+1)p'}\, dy\le\sig_{n-1} \intl_{1}^\infty s^{(d/2 -n+1)p'} (s^2 -1)^{n/2 -1}\,ds<\infty\]
provided $1\le p< 2(n-1)/d$.
\end{proof}
\begin{remark}\label{mjn} {\rm The condition $1\le p< 2(n-1)/d$ is sharp. If $p\ge 2(n-1)/d$, then  the function
\[
  f (x)=\frac{(x_{n+1}^2-1)^{(1-n/2)/p}}{(x_{n+1}+1)^{1/p}\, \log (x_{n+1}+1)}\]
 belongs to $L^{p}(\hn)$, however, the corresponding integral  (\ref{o8a5br})  diverges.
 }
\end{remark}

The next auxiliary statement, in which  $\hat f(\xi)\equiv \hat f(ka_tn_u \xi_0)$  is restricted to  $u=0$, plays an important role in derivation of the inversion formulas in Section 4.

\begin{lemma}  Given a function $\vp$ on $(1,\infty)$, let
\be\label{ju2} \psi(s)=2^{d/2}\sig_{d-1}\intl_1^s \frac{\vp (r)}{\sqrt{r^2 -1}}\, (s\!-\!r)^{d/2 -1}\,dr, \qquad s>1.   \ee
 Then
\be\label{ju1} \intl_{\bbr} \!\!\vp (\cosh t)\, e^{td/2}\, dt\!\!\intl_K \!\!\hat f(ka_t \xi_0)\, dk \!=\!\frac{1}{\sig_{n-1}}\, \intl_{\hn}\! \! f(x) \,\frac{\psi (x_{n+1})\, dx}{(x_{n+1}^2 \!-\!1)^{n/2 -1}}.\ee
It is assumed that $f$ and $\vp$ are good enough, so that  integrals in  (\ref{ju2}) and  (\ref{ju1}) exist in the Lebesgue sense.
\end{lemma}

\begin{proof} We denote the left-hand side of (\ref{ju1}) by $I$.  Then
\[ I=\intl_\bbr \vp (\cosh t)\, e^{td/2} \left [\,\intl_K   f(k (\cdot))\, dk\right]^\wedge \!\!( a_t \xi_0)\, \, dt.\]
The function in square brackets is zonal and we denote
\be\label{ju4} \intl_K  f(k x)\, dk=f_0 (x_{n+1}).\ee
Then, by Lemma \ref{zonal},
\bea I&=&2^{d/2 -1}\sig_{d-1}\intl_\bbr \vp (\cosh t)\, dt\intl_{{\rm cosh} t}^{\infty}\!\!f_{0}(s)\, (s\!-\!\cosh t)^{d/2 -1}\, ds\nonumber\\
&=&2^{d/2}\sig_{d-1} \intl_1^\infty \vp (r)\,\frac{dr}{\sqrt{r^2 -1}} \intl_r^\infty f_{0}(s)\,(s\!-\!r)^{d/2 -1}\, ds\nonumber\\
\label{ju3} &=& \intl_1^\infty f_{0}(s)\,\psi(s)\, ds.  \eea
Using (\ref{ju4}) and (\ref {kUUUPqs}), we continue:
\bea I&=& \intl_1^\infty \psi(s)\, ds\intl_K f(k (e_n\,\sinh r +s e_{n+1}))\, dk\nonumber\\
&=&\frac{1}{\sig_{n-1}}\,  \intl_0^\infty \psi (\cosh r)\, \sinh r\, dr\intl_{S^{n-1}}  f(\th\, \sinh r + e_{n+1} \cosh r)\, d\th\nonumber\\
&=&\frac{1}{\sig_{n-1}}\,  \intl_{\hn} f(x) \,\frac{\psi(x_{n+1})\, dx}{(x_{n+1}^2 -1)^{n/2 -1}},\nonumber\eea
as desired.
\end{proof}

\begin{example}\label{iio}
Let $\vp (r)=(r-1)^{\a/2 -1}\, \sqrt{r^2 -1}$, $\;\a >0$. Then
\bea
\psi(s)&=&2^{d/2}\sig_{d-1}\intl_1^s (s\!-\!r)^{d/2 -1}\,(r-1)^{\a/2 -1}\,dr\nonumber\\
&=&\frac{2^{d/2}\sig_{d-1}\, \Gam (\a/2)\,  \Gam (d/2)}{ \Gam ((\a+d)/2)}\, (s-1)^{(\a +d)/2-1}\nonumber\eea
and we have
\bea  &&\intl_{\bbr} e^{td/2}\,(\cosh t -1)^{\a/2 -1}\, |\sh t| \, dt\intl_K \hat f(ka_t \xi_0)\, dk \nonumber\\
\label{ju5}&&=c_1\, \intl_{\hn} \frac{(x_{n+1} -1)^{(\a+d-n)/2}}{(x_{n+1} +1)^{n/2 -1}}\, f(x) \,dx, \eea
\[c_1=\frac{2^{d/2}\sig_{d-1}\, \Gam (\a/2)\,  \Gam (d/2)}{ \sig_{n-1}\,\Gam ((\a+d)/2)}.\]
\end{example}

\section{Inversion Formulas}

In this section we obtain main results of the  paper. The proofs rely on the properties of hyperbolic convolutions and spherical means which are reviewed below.

\subsection{Hyperbolic Convolutions and Spherical Means}\label{pUUdu}

All  details related to this subsection can be found in \cite [Section 6.1.2]{Ru15} and \cite{Ru17}.

Given a measurable function $k$ on $[1, \infty)$, the
corresponding  hyperbolic convolution on $\hn$ is defined by
\be\label {tag 2.16-HYP} (K f) (x) = \int\limits_{\hn} k ([x, y]) f (y)  \, d y, \qquad x \in \hn. \ee
If this integral exists in the Lebesgue sense, then, by
 Fubini's theorem,
\be\label {tag 2.17-HYP} \int\limits_{\hn}\! k ([ x, y]) f (y) \, d y\!  =\!
\sigma_{n -1}\!  \int\limits^\infty_0\!  k (\ch r) \ (M_x f) (\ch r) \
\sh^{n -1} r \, d r, \ee
where $M_x f$ is the spherical mean
\be\label {2.21hDIF}
 (M_x f)(s) = {(s^2-1)^{(1-n)/2}\over \sigma_{n-1}}
 \intl_{\{y \in \bbh^n:\; [x, y]=s\}} f(y)\, d\sigma  (y),\quad s>1,\ee
$ d\sigma  (y)$
 being  the relevant induced  measure. We can also write (\ref{2.21hDIF}) in the ``more geometric'' form as
 \be\label {2.21hDIFbb}
 (M_x f)(\ch t)= \intl_K f(r_x k a_t x_{0})\, dk,\ee
where $r_x\in G$ takes $x_0$ to $x$ and $a_t$ is the matrix (\ref{mo3a1}).

\begin{lemma}\label {Lemma 2.1-HYP} {\rm (\cite [p. 370]{Ru15}, \cite[pp. 131-133]{Liz93})}. \\
Let $f \!\in \!L^p (\hn)$, $ 1\! \le \!p \!\le \!\infty$. Then
\be\label {2.21hDIFb}  \sup\limits_{s > 1} \| (M_{(\cdot)} f)(s) \|_p \le
\| f \|_p.\ee
  If $1 \le p < \infty$, then  $(M_x f)(s)$ is a continuous $L^p$-valued function of  $s\in [1,\infty)$ and
\be\label {2.XIFb} \lim\limits_{s \to 1} \| (M_{(\cdot)} f)(s) - f \|_p = 0.\ee  If $f \in C_0 (\hn)$, then $(M_x f)(s)$ is a continuous function of $(x,s)\in \hn \times (1,\infty)$ and
$(M_x f)(s)\to f(x)$
as $s\to 1$, uniformly on $\hn$.
 \end{lemma}

An important example of  convolutions (\ref{tag 2.16-HYP})  is the
 analytic family of the potential type operators
\be\label {BGBBQQ}
(Q^\a f)(x)\!=\!\zeta_{n,\a}\intl_{\hn}\!\! f(y)\, \frac{([x,y]-1)^{(\a -n)/2}}{([x,y]+1)^{n/2 -1}}\, dy, \ee
\[
\zeta_{n,\a}= \frac{\Gamma((n-\a)/2)}{2^{\a/2 +1}\,\pi^{n/2}\Gamma(\a/2)},  \qquad
Re\,\a >0, \quad \a-n\neq 0, 2, 4, \dots.\]
This analytic family naturally arises in \cite{Ru17} in the study of the horospherical transforms.

\begin{proposition}  \label {tag LLyHOR}   {\rm  \cite [p. 385]{Ru15} }
If $f \in L^p(\hn)$, $ 1\le p < \infty$, $ 0 < \a < 2(n-1)/p$, then $(Q^\a f)(x)$ exists
as an absolutely convergent integral for almost all $x\in\hn$.
\end{proposition}

 \begin{lemma}\label {LemTheor 4.4HYQQ}  {\rm  \cite [p. 386]{Ru15}} Let $f\in C^\infty_c(\hn)$,
 \[ D_\a = - \Delta_H-\a (2n-2-\a)/4, \qquad \a \ge 2,\]
 where $\Delta_H$ is the Beltrami-Laplace operator on $\hn$. If $\a-n\neq 0, 2, 4, \ldots$, then
\be\label{tag 4.8OFR1QQ}
D_\a Q^\a f = Q^{\a-2} f \qquad (Q^0 f=f).\ee
In particular, if $\a=2\ell$ is even, $2\ell-n\neq 0, 2, 4, \ldots $, and $\P_\ell(\Delta_H) =  D_{2}D_{4}\ldots D_{2\ell}$,  then
\be\label{tYUYUFR1} \P_\ell(\Delta_H)Q^{2\ell}f=f.\ee
\end{lemma}

We will need an extension of Lemma \ref {LemTheor 4.4HYQQ} to the case $\a=n$. For $f\in C^\infty_c(\hn)$, we define  $Q^n f$ as a limit
\bea (Q^n f)(x)\!\!&=&\!\!\lim\limits_{\a\to n }\zeta_{n,\a}\intl_{\hn}\!\! f(y)\, \frac{([x,y]-1)^{(\a -n)/2}-1}{([x,y]+1)^{n/2 -1}}\, dy\qquad\nonumber\\
\label{tZ44FR1} \!\!&=&\!\! \zeta'_{n}\intl_{\hn}\!\! f(y)\, \frac{\log ([x,y]\!-\!1)}{([x,y]\!+\!1)^{n/2 -1}}\, dy, \quad \zeta'_{n}=-\frac{2^{-1-n/2}}{\pi^{n/2}\, \Gam (n/2)}.\quad \qquad \eea

The following statements were proved in \cite{Ru15, Ru17}.

\begin{lemma}\label {Le76768QQ} Let $f\in C^\infty_c(\hn)$, $ D_n = - \Delta_H -n (n-2)/4$, $n\ge 2$. Then
\be\label{tYUY34}  D_n Q^n f= Q^{n-2} f +Bf\qquad (Q^{0} f=f),\ee
where
\be\label{tYUY341} (Bf)(x)=\zeta'_{n}\intl_{\hn} f(y)\frac{dy}{([x,y]\!+\!1)^{n/2 -1}}.\ee
\end{lemma}

\begin{lemma}\label {Le76768QQ1}  Let $f\in C^\infty_c(\hn)$,  $n> 2$. Then $(Bf)(x)$ is an eigenfunction of the Beltrami-Laplace operator $\Del_H$, so that
\be\label{tYUY342} - \Del_H Bf=\frac{n(n-2)}{4}\, Bf\ee
and
\be\label{tRRRYUY341}
D_n Bf=D_{n-2} Bf=0.\ee
\end{lemma}

\begin{proposition}\label{CoYYY5HY}  Let $f\in C^\infty_c(\hn)$, where $n$ is even.
If $n=2$, then
\be\label{tYUY343X}  - \Delta_H Q^2 f=f-\frac{1}{4\, \pi}\intl_{\bbh^2} f(y)\, dy.\ee
If $n\ge 4$, then
\be\label{tYUY343} \P_{n/2}(\Delta_H)\,Q^n f=f, \ee
\[ \P_{n/2}(\Delta_H) = (-1)^{n/2} \prod\limits_{i=1}^{n/2}   (\Delta_H+i (n-1-i)).  \]
\end{proposition}

\subsection{The Method of Mean Value Operators}

An idea of this inversion method is to average $\hat f(\xi)$ over all $d$-horospheres $\xi$ at a fixed positive distance from a given point $x\in \hn$. Inverting a simple Abel type fractional integral, we then obtain the spherical mean (\ref{2.21hDIF}) that gives $f(x)$ after passing to the limit according to Lemma \ref{Lemma 2.1-HYP}.

For $\xi= ka_tn_u \xi_0$, the relation $\xi\ni x_0$ is equivalent to $a_{-t}n_{-u} x_0 \in \xi_0$. By (\ref{horo coord1}),
\[a_{-t}n_{-u} x_0 =(u, 0, \ldots 0,\,-\sinh t+\frac{|u|^{2}}{2}e^{-t},\,\cosh t+\frac{|u|^{2}}{2}e^{-t}) \in \xi_0\]
(with $d$ zeros) if and only if $t=0$ and $u=0$. These two parameters contribute to the distance between $\xi$ and the origin $x_0$. We can work with one of them or with both. Suppose  $u=0$ and consider the mean value
\[\intl_K \hat f (r_x ka_t \xi_0) \, dk =\intl_K \hat f_x (ka_t \xi_0) \, dk , \qquad t>0.\]
Here  $r_x\in G$ is an arbitrary transformation satisfying  $r_x x_0=x$ and $f_x (y) =f(r_x y)$. Note that $r_x$ can be moved under the sign of the horospherical transform because the latter is $G$-equivariant.

We introduce the mean value operator
\be\label{duy} \check \vp_x (t)=\intl_K \vp (r_x k a_t \xi_0)\, dk, \qquad x\in \hn, \quad t\in \bbr.\ee

\begin{lemma} \label {NNNDDEER} If $\vp=\hat f$, then
\be\label {gyuyyhER} \check \vp_x (t)=c \,e^{-td/2}
\, (I^{d/2}_{-}  M_x f)(\ch t), \qquad c= 2^{d/2 -1}\, \sig_{d-1} \Gam (d/2),  \ee
where   $M_x f$ is the spherical mean (\ref{2.21hDIF}).
It is assumed  that the integral  on the right-hand side of (\ref{gyuyyhER}) exists in the Lebesgue sense.
\end{lemma}

\begin{proof} Fix $x\in \hn$ and let $f_x (y)= f(r_x y)$, $y\in \hn$. By  $G$-invariance,
\[
 \check \vp_x (t)=\intl_K \hat f (r_x\, k a_t \xi_0)\, dk = \Big [\intl_K f_x (ky)\, dk \Big ]^\wedge (a_t \xi_0).\]
 The function $y \to  \intl_K f_x (ky)\, dk$ is zonal, so that there is a single-variable function $f_{0,x}(\cdot)$ such that
\be\label {gyuyyhYPER} f_{0,x} (y_{n+1})= \intl_K f_x (ky)\, dk.\ee
By (\ref{o8a5br}) with $u=0$,
\[
\check \vp_x (t)=2^{d/2 -1}\, \sig_{d-1}\, e^{-td/2}\,(I_-^{d/2}  f_{0,x})(\ch t), \]
where, by (\ref {gyuyyhYPER}),
\[
f_{0,x}(s)=\intl_K f_x (k(e_n \,\sqrt{s^2 -1} +e _{n+1}\, s ))\, dk=(M_x f)(s).\]
This completes the proof.
\end{proof}

We denote
\be\label {next aim} g_x (s)=(M_x f)(s), \qquad \psi_x (r)\!=\!c^{-1} e^{td/2} (\hat f)^\vee_x(t)\big |_{t= {\rm cosh}^{-1} r}\,,\ee
$c$ being the same as in (\ref{gyuyyhER}). Then (\ref{gyuyyhER}) can be written as
\be\label{uutrrhER}(I^{d/2}_{-}g_x)(r)=\psi_x (r).\ee
By Lemma \ref {NNNDDEER}, to reconstruct $f$, we first need to find $g_x (s)=(M_x f)(s)$ from the Abel equation (\ref{uutrrhER}) by using the tools of fractional differentiation \cite{Ru15}. Then  $f$  will be obtained as a limit $f(x)=\lim\limits_{s\to 1} (M_x f)(s)$ in accordance with Lemma \ref{Lemma 2.1-HYP}.

The proof of the following statements is omitted because it is an almost verbatim copy of the reasoning from \cite{Ru15} for $d=n-1$ .  Everywhere in these statements, we assume  $1\le d\le n-1$, $f\in C_\mu (\hn)$, $\mu>d/2$, or $f\in L^p (\hn)$, $1\le p<2(n-1)/d$.

\begin{lemma}  {\rm (cf. Corollary 6.77 in \cite{Ru15})}
 The integral $(I^{d/2}_{-}g_x)(r)$ exists in the Lebesgue sense for almost all $r> 1$ and all $x\in \hn$. If, moreover, $d\ge 2$, then $I^{d/2}_{-}g_x$ is a continuous function on $(1,\infty)$  for all $x\in \hn$.
\end{lemma}

\begin{lemma}\label{jjYBYB} {\rm (cf. Corollary 6.78 in \cite{Ru15})}  Let $ g_x (s)=(M_x f)(s)$.
 If $I^{d/2}_{-}g_x=\psi_x$,   as in (\ref{uutrrhER}), then
 \be\label {IO8kVF}g_x (s)= (\Cal D^{d/2}_{-} \psi_x)(s)\qquad \forall \, s>1,\ee
 where the  derivative $\Cal D^{d/2}_{-} \psi_x$ is defined  as follows.

 \noindent {\rm (i)} If $d$ is even, $d=2m$, then
\be\label{uimuyHOR}
(\Cal D^{d/2}_{-} \psi_x)(s)= (-1)^{m } \left (\frac{d}{ds}\right )^m \psi_x (s). \ee

\noindent {\rm (ii)} If $d$ is odd, $d=2m-1$, then
\bea\label{uimuyHOR2}
(\Cal D^{d/2}_{-} \psi_x)(s)&=& (-1)^{m }   s^{1/2}\, \left (\frac{d}{ds}\right )^m \left [ s^{m-1/2}\, I^{1/2}_{-} \,s^{-m}\, \psi_x \right ],\\
\label{frYUT2m}  &=&(-1)^{m}  s^{1/2}\,  \frac{d}{ds}\left [ s^{1/2}\, I^{1/2}_{-} \,s^{-1 }\,\psi_x^{(m-1)}\right ].\eea
 The equalities (\ref{uimuyHOR})-(\ref{frYUT2m})  hold for all $x\in \hn$, if  $f\in C_\mu (\hn)$, and for almost all $x\in \hn$, if  $f\in L^p(\hn)$.
\end{lemma}

\begin{theorem} \label{jOOOthERC2} {\rm (cf. Theorem 6.79 in \cite{Ru15})} Let
\[ \vp =\hat f, \qquad \psi_x (r)\!=\!c^{-1} e^{td/2} \check\vp_x(t)\big |_{t= {\rm cosh}^{-1} r}, \quad c= 2^{d/2 -1}\, \sig_{d-1} \Gam (d/2).\]
Then
\be\label{jOOOthERC3}  f(x)=\lim\limits_{s\to 1}(\Cal D^{d/2}_{-} \psi_x)(s),  \ee
where $(D^{d/2}_{-} \psi_x)(s)$ is defined by (\ref{uimuyHOR})-(\ref{frYUT2m}).
 The limit in  (\ref{jOOOthERC3}) is uniform for $f\in C_\mu (\hn)$  and is understood in the $L^p$-norm if $f\in L^p(\hn)$.
\end{theorem}

\subsection{Inversion of the Horospherical  Transforms by Polynomials of the Beltrami-Laplace Operator}\label{jOMMMC2}

\subsubsection{Local Inversion Formulas for $d$ Even}

We consider the mean value operator (\ref{duy}) with $t=0$ and denote
\be\label{hot2}
\check\vp (x)=\intl_K \vp (r_x k\xi_0)\, dk.
\ee
This operator integrates a function $\vp$ on $\Xi_d$ over all $d$-horospheres passing through a given point $x\in \hn$.
 The next lemma can be considered as a modification of  Lemma \ref{NNNDDEER} corresponding  to $t=0$. It
 establishes an important connection between the $d$-horospherical transform, the mean value operator (\ref{hot2}) and the analytic family (\ref{BGBBQQ}).  This lemma is a horospherical analogue of the celebrated Fuglede result for $d$-plane Radon-John transforms \cite{F}. Similar statements are known for all totally geodesic Radon transforms on constant curvature spaces \cite{H11, Ru15, Ru02b, Ru02c}.

\begin{lemma} \label {MMM-HYP3horT}  The following equality holds provided that either side of it exists in the Lebesgue sense:
\be\label {taRRRrayhor} (\hat f)^\vee(x)= c\, (Q^d f)(x), \qquad c=\frac{2^d\, \pi^{d/2}\, \Gam (n/2)}{\Gam ((n-d)/2)}. \ee
\end{lemma}
\begin{proof} Setting $t=0$ in (\ref{gyuyyhER}), and using (\ref{tag 2.17-HYP}), we obtain
\bea (\hat f)^\vee(x)&=& 2^{d/2 -1}\, \sig_{d-1} \intl_1^\infty (s-1)^{d/2 -1}  (M_x f)(s)\, ds\nonumber\\
&=&   2^{d/2 -1}\, \sig_{d-1} \intl_0^\infty (\ch r-1)^{d/2 -1}  (M_x f)(\ch r)\,\sh r\,  dr\nonumber\\
&=&    \frac{\Gamma((n-d)/2)}{2^{d/2 +1}\,\pi^{n/2}\Gamma(d/2)} \intl_{\hn}\!\! f(y)\, \frac{([x,y]-1)^{(d -n)/2}}{([x,y]+1)^{n/2 -1}}\, dy=c\, (Q^d f)(x). \nonumber\eea
\end{proof}

Lemmas \ref {MMM-HYP3horT} and \ref {LemTheor 4.4HYQQ} imply the following inversion result.

\begin{theorem}\label{ThHORYP}  Let $\varphi = \hat f$, $f \in C^\infty_c(\hn)$, $1\le d\le n-1$.
If $d$ is even, then
\be\label {tag 1.6hnYO1} f=c \,\P(\Delta_H)\,\check\varphi, \ee
where
\[ \P (\Delta_H) = \prod\limits^{d/2}_{i=1} \big[\Delta_H + i(n-1-i)\big], \quad
  c=\frac{(-1)^{d/2}\,\Gam ((n-d)/2)}{2^d\, \pi^{d/2}\, \Gam (n/2)}.\]
\end{theorem}

\subsubsection{Inversion Formulas for Arbitrary $d$}

If $d$ is odd then a local inversion formula, like (\ref{tag 1.6hnYO1}), is unavailable in principle. Both even and odd cases   can be treated in the framework of a certain analytic family of operators generalizing the mean value operator $\vp \to \check \vp$. This approach is inspired by our previous works; cf. \cite[Theorem 1.2]{Ru02b}, \cite[Theorem A]{Ru02c}, \cite[Theorem 4.13]{Ru17}.

We replace $f$ in (\ref{ju5}) by the shifted function $f_x (y)=f(r_x y)$ where $x\in \hn$ is a new exterior variable and $r_x\in G$ satisfies  $r_x x_0=x$. Denote
\[
 h_\a (t)= e^{td/2}\,(\cosh t -1)^{\a/2 -1}\,  |\sh t|, \qquad \a>0,\]
 and write   (\ref{ju5}) as
 \bea \intl_{K\times \bbr} \hat f(r_x ka_t \xi_0)\, h_\a (t)\, dk dt&=& c_1\, \intl_{\hn} \frac{(y_{n+1} -1)^{(\a+d-n)/2}}{(y_{n+1} +1)^{n/2 -1}}\, f_x (y) \,dy\nonumber\\
\label{ju1u}&=&c_1\, \intl_{\hn} f(y)\, \frac{([x,y]-1)^{(\a+d -n)/2}}{([x,y]+1)^{n/2 -1}}\, dy, \eea
\[
c_1=\frac{2^{d/2}\sig_{d-1}\, \Gam (\a/2)\,  \Gam (d/2)}{ \sig_{n-1}\,\Gam ((\a+d)/2)}, \qquad \a>0.\]
In particular, for $\a=n-d$,
\be\label{poo} \intl_{K\times \bbr} \hat f(r_x ka_t \xi_0)\, h_{n-d} (t)\, dk dt=\tilde c_1\,  \intl_{\hn}\, \frac{ f(y)\, dy}{([x,y]+1)^{n/2 -1}}, \ee
 \[  h_{n-d} (t)\!=\!e^{td/2}  (\cosh t-1)^{(n-d)/2 -1}\,  |\sh t|, \]
 \be\label{poog}\tilde c_1 \!=\!2^{d/2}\, \pi^{(d-n)/2}\, \Gam ((n-d)/2).\ee

An expression in (\ref{ju1u}) is a constant multiple of the convolution  $Q^{\a +d} f$; cf. (\ref{BGBBQQ}). Changing normalization,
we obtain the following statement.
\begin{lemma} \label{duJJYPhor} Let
\be\label{jusg} (\overset *{\frH}{}^\a \vp)(x)=c_{\a}  \intl_{K\times \bbr} \vp(r_x ka_t \xi_0)\, h_\a (t)\, dk dt,\ee
\[  c_{\a}=\frac{\Gam ((n\!-\!\a\!-\!d)/2)}{2^{\a/2 +d+1}\, \pi^{d/2}\,\Gam (\a/2)\,  \Gam (n/2)},\qquad  \a>0, \quad \a+d-n\neq 0, 2, 4,, \ldots\, .\]
Then
\be\label{jusg1}  \overset *{\frH}{}^\a \hat f= Q^{\a +d} f\ee
provided that the integral on the right-hand side exists in the Lebesgue sense.
\end{lemma}

The next proposition shows that the mean value operator (\ref{hot2}) is  a constant multiple of the limit  of the operators (\ref{jusg}) as   $\a\to 0$.

\begin{proposition}\label{juusg1}
If  $\vp$ is a  compactly supported continuous function on $\Xi_d$, then
\be\label{jusg2}
\lim\limits_{\a\to 0} \overset *{\frH}{}^\a \vp = c^{-1}\, \check \vp\ee
where $c$ is a constant from (\ref {taRRRrayhor}).
\end{proposition}
\begin{proof} We write (\ref{jusg}) as
\bea
(\overset *{\frH}{}^\a \vp)(x)&=&c_{\a}  \intl_{\bbr} e^{td/2}\,  |\sh t|^{\a-1} (\cosh t +1)^{1-\a/2}\, dt\intl_K \vp(ka_t \xi_0)\, dk \nonumber\\
\label{kip} &=&\frac{1}{\gam_{1} (\a)}\, \intl_{\bbr}   |t|^{\a-1}\, \Om_\a (t)\, dt,\eea
where
\[
\gam_{1} (\a)=\frac{2^\a\, \pi^{1/2} \, \Gam (\a/2)}{\Gam ((1-\a)/2)}, \]
\[ \Om_\a (t)=\frac{c_{\a}}{\gam_{1} (\a)}\,
\left (\frac {\sh t}{t}\right )^{\a-1} (\cosh t +1)^{1-\a/2}\,e^{td/2}\, \intl_K \vp(ka_t \xi_0)\, dk.\]
By the well-known properties of Riesz kernels (see, e.g. \cite[Lemma 3.2]{Ru15}), the limit of the expression (\ref{kip})  as   $\a\to 0$ is
\[
\Om_\a (0)\big |_{\a=0}= c^{-1}\, \check \vp,\]
where $c$ is a constant from (\ref {taRRRrayhor}).
\end{proof}

We will need  an analogue of Lemma \ref{duJJYPhor} for the case $\a=n-d$, which was excluded in (\ref{jusg1}) because of the pole  of the gamma function in $c_a$. Starting from (\ref{jusg}), we
 define
\be\label{tag 4.6OFR1HOOO8}
(\overset *{\frH}{}^{n-d} \vp)(x)=c_{n,d}\intl_{K\times \bbr} \vp(r_x ka_t \xi_0)\, \tilde h_{n-d} (t)\, dk dt,\ee
where
\[c_{n,d}=-\frac{1}{2^{(n+d)/2  +1}\,\pi^{d/2}\,\Gam (n/2)\,\Gam ((n-d)/2)},\]
\[\tilde h_{n-d} (t)=e^{td/2} \,  |\sh t| \, (\cosh t-1)^{(n-d)/2 -1}\, \log(\cosh t-1).\]
We  also use the notation $Q^{n}$ and $ \tilde c_1$ from (\ref{tZ44FR1}) and (\ref{poog}), respectively.

\begin{proposition} \label{duJJYP1} Let $\vp = \hat f$, $f \in C^\infty_c(\hn)$, $1\le d\le n-1$. Then
\be\label{tag 4.6OFR1HOOO}
\overset *{\frH}{}^{n-d} \vp = Q^{n} f+ \Phi,\ee
where
\bea\label{tMJOO8}
\Phi (x)&=& \gam_{n,d}\,\intl_{\hn} \frac{f(y)\,dy}{([x,y]+1)^{n/2 -1}}\\
&=& \tilde\gam_{n,d}\,\intl_{K\times \bbr} \vp(r_x ka_t \xi_0)\, \tilde h_{n-d} (t)\, dk dt,\nonumber\eea
\[\gam_{n,d} =\frac{\psi(n/2)-\psi ((n-d)/2)}{2^{n/2+1}\,\pi^{n/2}\,\Gam (n/2)}, \qquad \psi(z)=\frac{\Gam' (z)}{\Gam (z)},\]
\[
\tilde\gam_{n,d}=\frac{\gam_{n,d}}{\tilde c_1}=\frac{\psi(n/2)-\psi ((n-d)/2)}{2^{(n+d)/2+1}\,\pi^{d/2}\,\Gam (n/2)\,\Gam ((n-d)/2)}.\]
\end{proposition}
\begin{proof}
  For  $\a  \neq n-d$, but close to $n-d$, we  write (\ref{jusg1}) as
\bea
 &&c_{\a}  \intl_{\Xi_d} \intl_{K\times \bbr} \hat f(r_x ka_t \xi_0)\, e^{td/2}\nonumber\\
   &&\times [(\cosh t\!-\!1)^{\a/2 -1}\! - \!(\cosh t\!-\!1)^{(n-d)/2 -1}]\,  |\sh t|\, dk dt+ c_{\a} I_1\nonumber\\
 &&=\z_{n,\a+d} \intl_{\hn}  f(y)\, \frac{([x,y]-1)^{(\a+d-n)/2} -1}{([x,y]+1)^{n/2-1}}\, dy +\z_{n,\a+d}\,I_2,\nonumber\eea
where
 \[
 I_1= \intl_{K\times \bbr} \hat f(r_x ka_t \xi_0)\, e^{td/2} (\cosh t-1)^{(n-d)/2 -1}\,  |\sh t|\, dk dt,\]
 \[ I_2=\intl_{\hn}  \frac{ f(y)\,dy}{([x,y]+1)^{n/2-1}}.\]
 By (\ref{poo}),  $I_1=\tilde c_1 I_2$. Moving $c_{\a} I_1= c_{\a}\tilde c_1 I_2$ to the right-hand side and
passing to the limit  as $\a \to n-d$, we obtain (\ref{tag 4.6OFR1HOOO}).
\end{proof}

Now we can formulate  the  inversion result for $\hat f$ in the most general form.

\begin{theorem}\label{ThHORYP1}  Let $\varphi = \hat f$, $f \in C^\infty_c(\hn)$,
\be\label {tag 1HORnYO} \P_\ell (\Delta_H) = (-1)^{\ell}\prod\limits^\ell_{i=1} \big[\Delta_H + i(n-1-i)\big], \qquad  \ell \in\bbn.\ee

\noindent {\rm (i)}  If $n$ is odd, then
\be\label {tag 1.6hnYO} f= \P_\ell (\Delta_H) \overset *{\frH}{}^{2\ell-d} \vp \qquad \forall \ell \ge d/2.\ee

\noindent {\rm (ii)} \ If  $n=2$, then
\be\label {tag 1HORnYOA} f= -\Delta_H \overset *{\frH}{}^1 \vp +\frac{1}{4\, \pi}\intl_{\bbr} \vp (a_t\xi_0)\, dt.\ee

\noindent {\rm (iii)} \ If  $n=4,6, \ldots $, then
\be\label {tag 1HORnYOAi} f= \P_{n/2}(\Delta_H)  \overset *{\frH}{}^{n-d} \vp.  \ee
\end{theorem}

\begin{proof}  If $n$ is odd, then (\ref{jusg1}) with $\a=2\ell -d >0$ gives $\overset *{\frH}{}^{2\ell-d} \vp= Q^{2\ell} f$  and the result follows by (\ref{tYUYUFR1}). If $2\ell -d =0$, the desired statement was obtained in Theorem \ref{ThHORYP}; cf. Proposition \ref{juusg1}.

 If $n=2$, then, by (\ref{tag 4.6OFR1HOOO}),
 \[\stackrel{*}{\fr{H}}{\!}^1 \vp = Q^{2} f+ \Phi, \qquad \Phi= \gam \intl_{\hn} f(y)\,dy, \quad  \gam=\frac{\psi (1) -\psi(1/2)}{4\pi}.\]
Applying  $-\Delta_H$ to both sides of this equality, owing to (\ref{tYUY343X}) and (\ref{o8a5btg2}), we obtain
\[
-\Delta_H \stackrel{*}{\fr{H}}{\!}^1 \vp=f-\frac{1}{4\, \pi}\intl_{\bbh^2} f(y)=f-\frac{1}{4\, \pi}\intl_{\bbr} \vp (a_t\xi_0)\, dt.\]
 This gives (\ref{tag 1HORnYOA}).

If $n=4,6, \ldots $, then, by (\ref{tag 4.6OFR1HOOO}) and (\ref{tYUY341}),
\bea
\P_{n/2}(\Delta_H) \stackrel{*}{\fr{H}}{\!}^{n-d} \vp&=&\P_{n/2}(\Delta_H) Q^n f +\P_{n/2}(\Delta_H) \Phi \nonumber\\
&=&f + \frac{ \gam_{n,d}}{\z_n'}\, \P_{n/2}(\Delta_H)\, Bf.\nonumber\eea
By Lemma \ref{Le76768QQ1}, $\P_{n/2}(\Delta_H)\, Bf=D_2\cdots D_{n-2}Bf=0$. Hence, we are done.
\end{proof}

\section{Conclusion}
In the present paper, we studied  the Radon-like transform  over $d$-dimensional horospheres in the hyperbolic space $\hn$ for any $0<d<n$. Our main concern was explicit inversion formulas for this transform acting on continuous and $L^p$ functions. The set $\Xi_d$ of all $d$-horospheres was defined in the group-theoretic terms as a $G$-orbit of the basic $d$-horosphere lying in the cross-section of $\hn$ by a fixed $(d+2)$-dimensional coordinate  plane containing the $x_{n+1}$-axis. One can give an alternative, ``more geometric" definition of the set of  $d$-horospheres as the set $\tilde\Xi_d$ of all cross-sections of the hyperboloid $\hn$ by $(d+1)$-dimensional   planes  having the orthogonal $(n-d)$-dimensional normal frames lying in the asymptotic cone $\Gam$.

Note also that $ \dim \Xi_d= (n-d)(d+2) -1$. This formula is a consequence of the isomorphism $\Xi_d \sim G/ M_d N_d$ (see Proposition \ref {plu}) and known dimensions\footnote{The formula for the dimension of $G$ is immediate, e.g.,  from the Iwasawa decomposition $G=KAN$.}
\[\dim G=\dim SO_0 (n,1)=\frac{n(n+1)}{2}, \qquad \dim O(n)= \frac{n(n-1)}{2}.\]
It follows that $ \dim \Xi_d > \dim \hn=n$ if and only if $1\le d< n-1$.

The following open problems arise.

\noindent{\bf Problem 1.} Investigate the relationship between $\Xi_d$ and $\tilde\Xi_d$.

\noindent{\bf Problem 2.} Reduce the overdetermindness  of the $d$-horospherical Radon transform in the case $d< n-1$.

In Problem 2, our aim is to define an $n$-dimensional submanifold  $\Xi_{d,0}$  of $\Xi_d$, so that  a function $f(x)$  could be recovered from its $d$-horospherical transform $\hat f(\xi)$ when the values of  $\hat f(\xi)$ are  known  only for $\xi\in \Xi_{d,0} $. Inegral-geometric problems of this kind amount to I.M. Gelfand \cite {Gelf60}. See also \cite{Ru15a}  and references therein.

We plan to address these problems in the future.

\end{document}